\documentclass[12pt]{amsart}

\usepackage{amsfonts,amsmath,latexsym,amssymb,verbatim,amsbsy,times}
\usepackage{amsthm, pdfsync}
\usepackage{pstricks}
\theoremstyle{plain}
\newtheorem{THEOREM}{Theorem}[section]

\newtheorem{lemma}[THEOREM]{Lemma}
\newtheorem{proposition}[THEOREM]{Proposition}

\theoremstyle{definition}

\theoremstyle{remark}

\newcommand{\Z}{\ensuremath{\mathbb{Z}}}   %%% integers
\newcommand{\R}{\ensuremath{\mathbb{R}}}   %%% reals
\newcommand{\T}{\ensuremath{\mathbb{T}}}   %%% torus

\def \a {\alpha}

\def \d {\delta}

\def \l {\lambda}

\def \n {\nabla}

\def \ex {\vec{e}_1}
\def \ey {\vec{e}_2}

\def \< {\langle}
\def \> {\rangle}
\def \p {\partial}
\def \ra {\rightarrow}

\newcommand{\der}[2]{(#1 \cdot \nabla) #2}

\newcommand{\tri}[3]{#1 \otimes #2 : \n #3}

\begin{document}

\title[Ill-posedness in Besov spaces]{Ill-posedness for subcritical hyperdissipative Navier-Stokes equations
in the largest critical spaces}

\author{A. Cheskidov}
\thanks{The work of A. Cheskidov is partially supported by NSF grant DMS--1108864}
\address[A. Cheskidov and R. Shvydkoy]
{Department of Mathematics, Stat. and Comp. Sci.\\
M/C 249,\\
      University of Illinois\\
      Chicago, IL 60607}
\email{acheskid@math.uic.edu}

\author{R. Shvydkoy}
\thanks{The work of R. Shvydkoy was partially supported by NSF grants
DMS--0907812 and DMS--1210896}

\email{shvydkoy@math.uic.edu}

\begin{abstract}
We study the incompressible Navier-Stokes equations with a fractional Laplacian and prove
the existence of discontinuous Leray-Hopf solutions in the largest critical space with arbitrarily small initial data.
\end{abstract}


\maketitle
\section{Introduction}
In this paper we study the supercritical 3D Navier-Stokes equations with a fractional power of the Laplacian
\begin{equation} \label{NSE}
\left\{
\begin{aligned}
&\p_t u  + (u \cdot \nabla)u + \nabla p = - \nu  (-\Delta)^{\alpha} u , \qquad x \in \T^3, t\geq 0,\\
&\nabla \cdot u =0,\\
&u(0)=u_0,
\end{aligned}
\right.
\end{equation}
where the velocity $u(x,t)$ and the pressure $p(x,t)$ are unknowns,
$u_0 \in L^2(\T^3)$ is the initial condition,
$\nu>0$ is the kinematic  viscosity coefficient of the fluid, and $\alpha>0$. The case $\alpha=1$ corresponds 
to the classical Navier-Stokes equations, which has been studied extensively for decades. We refer to \cite{CF, temam}
for the classical theory for these equations. In the case $\alpha\geq5/4$ the equations are well-posed, as the dissipative term simply dominates the nonlinear
term. Moreover, the global regularity is known even in a slightly supercritical
case, i.e., when logarithmic corrections to the Fourier multiplier
of the dissipative term are present (see \cite{tao,cs-kolm}). However, a finite time blow up of solutions to \eqref{NSE} remains a possibility for $\alpha<5/4$ due
to a supercritical nature of the equations. Nevertheless,
a partial regularity result \cite{CKN} has been established in the supercritical case $\alpha=1$, later extended
to $\alpha\in(1,5/4)$ in \cite{KP}. There are also various regularity criteria in the case $\alpha=1$, most of
which are of Ladyzhenskaya-Prodi-Serrin type \cite{foias-reg, ladyzhenskaya, prodi, serrin, ESS, cs-reg, cs-kolm},
 which can also be extended to $\alpha\in(1,5/4)$.

One of the open questions studied extensively is whether solutions bounded
in the largest critical case ($\dot{B}^{-1}_{\infty,\infty}$ for $\alpha=1$) are regular.
A positive answer to this question would extend the famous $L_t^\infty L_x^3$
result due to Iscauriaza, Seregin, and \v{S}ver\'ak \cite{ESS}. In addition,
the best small initial result for the 3D NSE, due to Koch and Tataru \cite{KT},
is in the space $BMO^{-1}$, and it is not known either if its extension to the
$B^{-1}_{\infty,\infty}$ is possible.

In view of these problems two ``negative'' results have been obtained in the space $\dot{B}^{-1}_{\infty,\infty}$. First, Bourgain and Pavlovic \cite{bp} proved that that there are solutions
to the 3D NSE equations, with arbitrary small initial data in $\dot{B}^{-1}_{\infty,\infty}$ that become arbitrarily large in $\dot{B}^{-1}_{\infty,\infty}$ in arbitrarily
small time. Second, Leray-Hopf solutions with arbitrary small initial data, but discontinuous in $B^{-1}_{\infty,\infty}$ were obtained in \cite{cs-ill}.

The largest critical space for the fractional NSE \eqref{NSE} is $\dot{B}^{1-2\alpha}_{\infty,\infty}$.
Recently Yu and Zhai \cite{YZ} obtained a small initial data result in this space
in the hypodissipative case $\alpha \in (1/2, 1)$. Heuristically, the hypodissipative
NSE behaves better because it is closer to the fractional heat semigroup
in critical spaces. In the hyperdissipative case it is therefore natural to expect ill-posedness results of the type mentioned above. Indeed, in this paper we demonstrate this in the case $\a \in [1,5/4)$ by constructing a Leray-Hopf solution with arbitrarily small initial data, which is discontinuous
in the critical Besov space $B^{1-2\alpha}_{\infty,\infty}$. It is thus a direct extension of our previous result stated in \cite{cs-ill}. The method breaks down either when $\a$ passes beyond the value of $1$, which is consistent with the result of Yu and Zhai, and at $5/4$ and beyond, which is consistent with the global regularity in that range.

We now fix our notation. We assume periodic boundary conditions in all 3 dimensions, so $\T^3$ will denote the 3D torus, while $|\cdot |_p$, $p\geq 1$, denotes the $L^p$-norm in $\T^3$. We let $\hat{f}$ and $\check{f}$ stand for the forward and, respectively, inverse Fourier transforms on the torus. The Fourier multiplier with symbol $|\xi|^\a$, where $\xi$ stands for the frequency vector and $\a >0$, is denoted by $|\n|^\a$. The fractional Laplacian operator $(-\Delta )^{\a}$ has symbol $|\xi|^{2\a}$. We write $p(\xi) =\mathrm{id} - |\xi|^{-2} \xi \otimes \xi$, $\xi \neq 0$, $p(0) = \mathrm{id}$, for the symbol of the Leray-Hopf projection on the divergence-free fields. We fix notation for the dyadic a-dimensional wavenumbers $\l_q = 2^q$. We use extensively the classical dyadic decomposition throughout: $u = \sum_{q \geq 0} u_q$, where $u_q$ is the Littlewood-Paley projection with the Fourier support contained in $\{ \l_{q-1} < |\xi| < \l_{q+1}\}$.  The definitions are standard and can be found in the references above. We often will be using the extended projection defined by $\tilde{u}_q = u_{q-1} + u_q + u_{q+1}$, $q \geq 1$, and projection onto the dyadic ball, $u_{\leq q} = \sum_{p=0}^q u_p$. Thus, $\tilde{u}_q$ is supported on $\{ \l_{q-2} < |\xi| < \l_{q+2}\}$ and we have the identity
\begin{equation}\label{extended}
\int_{\T^3} u \cdot u_q\ dx = \int_{\T^3} \tilde{u}_q \cdot u_q\ dx. 
\end{equation}
With the Littlewood-Paley decomposition we define Besov spaces $B^s_{r,\infty}$, $s \in \R$, $r \geq 1$ by requiring
\[
\| u\|_{B^s_{r,\infty}} = \sup_{q \geq 0} \l_q^{s} \| u_q\|_r <\infty.
\]
We will frequently refer to Bernstein's inequalities, which state that for all $1\leq r < r' \leq \infty$, and in three dimensions, one has
\[
|u_q |_{r'} \lesssim \l_q^{3(1/r - 1/r')} | u_q |_{r},
\]
where here and throughout $\lesssim$ denote inequality up to an absolute constant.
Finally, let $\ex$, $\ey$, etc., stand for the vectors of the standard unit basis.

\section{Ill-posedness of NSE}

\def \lq {\l_{q_j}}
\def \lqm { \l_{q_j - 1} }
\def \lqk { \l_{q_k} }

The Navier-Stokes equation with a fractional power of the Laplacian is given by
\begin{equation}\label{nse}
u_t + \der{u}{u} = -\nu (-\Delta)^{\alpha} u - \n p.
\end{equation}
Here $u$ is a three dimensional divergence free field on $\T^3$, and $\a\in[1,5/4)$. Let us recall that for every field $U\in L^2(\T^3)$ there exists a weak solution $u \in C_w([0,T);L^2) \cap L^2([0,T);H^1)$ to \eqref{nse} such that the energy inequality
\begin{equation}\label{enineq}
|u(t)|_2^2 + 2\nu \int_0^t ||\n |^{\alpha} u(s)|^2_2 ds \leq |U|_2^2,
\end{equation}
holds for all $t>0$ and $u(t) \ra U$ strongly in $L^2$ as $t\ra 0$. In what follows we do not actually use inequality \eqref{enineq} which allows
us to formulate a more general statement below in Proposition~\ref{p-nse}.

Let us choose a strictly increasing sequence $\{q_i\} \in \mathbb{N}$ with elements sufficiently far apart so that at least
$\l_{q_i}^{2\a}  \l_{q_{i+1}}^{4\a-5} < 1$. We consider the following lattice blocks:
\begin{align*}
A_j & = \left[\frac{9}{10} \lq, \frac{11}{10} \lq \right] \times \left[- \frac{1}{10} \lq, \frac{1}{10}\lq \right]^2 \cap \Z^3 \\
B_j & = \left[-\frac{1}{10} \lqm, \frac{1}{10} \lqm \right]^2 \times \left[\frac{9}{10} \lqm, \frac{11}{10} \lqm \right] \cap \Z^3\\
C_j& = A_j + B_j \\
A_j^*& = -A_j, \ B_j^* = - B_j, \ C_j^* = -C_j.
\end{align*}
Thus, $A_j$, $C_j$ and their conjugates lie in the $q_j$-th shell, while $B_j$, $B_j^*$ lie in the adjacent $(q_j - 1)$-th shell. The particular choice of scaling exponents $9/10,11/10$, etc., is unimportant as long as the blocks fit into the their respective shells. Let us denote
$$
\ex(\xi) = p(\xi) \ex, \quad \ey(\xi) = p(\xi) \ey.
$$
We now define the initial condition field to be the following sum
\begin{equation}
U = \sum_{j \geq 1} (U_{q_j} + U_{q_j - 1} ),
\end{equation}
where the components, on the Fourier side, are
$$
\widehat{U_{q_j}}(\xi) =\lq^{2\alpha-4} \left( \ey(\xi) \chi_{A_j \cup A_j^*} + i (\ey(\xi) - \ex(\xi))\chi_{C_j} - i (\ey(\xi) - \ex(\xi))\chi_{C_j^*}    \right),
$$
and
$$
\widehat{U_{q_j - 1}}(\xi) = \lq^{2\alpha-4}  \ex(\xi) \chi_{B_j \cup B_j^*}.
$$
By construction, $\hat{U}(-\xi) = \overline{\hat{U}(\xi)}$, which ensures that $U$ is real. Since $U$ has no modes in the $(q_j + 1)$-st shell, then the extended Littlewood-Paley projection of the $j$-th component has the form
$\tilde{U}_{q_j}= U_{q_j - 1} + U_{q_j}$.

\begin{lemma}\label{ubesov}
We have $U \in B^{1+\frac{3}{r}-2\a}_{r,\infty}$, for any $1<r\leq \infty$.
\end{lemma}
\begin{proof}
We give the estimate only for one block, the other ones being similar. Using boundedness of the Leray-Hopf projection, we have, for all $1<r<\infty$,
$$
| \lq^{2\a-4} (\ey(\cdot) \chi_{A_j})^{\vee} |_r \lesssim \lq^{2\a-4} | (\chi_{A_j})^\vee |_r.
$$
Notice that by construction,
\[
| (\chi_{A_j})^\vee(x_1,x_2,x_3) | = | D_{(c+1)\lq } (x_1) D_{c\lq}(x_2) D_{c\lq}(x_3) |.
\]
where $D_N$ denotes the Dirichlet kernel. Hence,
\[
| (\chi_{A_j})^\vee |_r \leq |D_{(c+1)\lq }|_r | D_{c\lq} |_r^2.
\]
By a well-known estimate, we have $|D_N|_r \leq N^{1-\frac{1}{r}}$ (c.f. \cite{G}). Putting the above estimates together implies the desired inclusion in $B^{1+3/r-2\a}_{r,\infty}$. In the case $r = \infty$ we simply use the triangle inequality to obtain
$$
|U_{q_j}|_\infty \lesssim \lq^{2\a-1}.
$$
\end{proof}
Let us now examine the trilinear term. We will use the following notation for convenience
\begin{equation}
u\otimes v : \n w = \int_{\T^3} v_i \p_i w_j u_j dx.
\end{equation}

Using the antisymmetry we obtain
\begin{align*}
\tri{U}{U}{U_{q_j}}& = \sum_{k \geq j+1} \tri{\tilde{U}_{q_k}}{\tilde{U}_{q_k}}{U_{q_j}} +
\tri{\tilde{U}_{q_j}}{\tilde{U}_{q_j}}{U_{q_j}} \\
&+ \tri{U_{\leq q_{j-1}}}{\tilde{U}_{q_j}}{U_{q_j}} + \tri{\tilde{U}_{q_j}}{U_{\leq q_{j-1}}}{U_{q_j}} \\
&= \sum_{k \geq j+1} \tri{\tilde{U}_{q_k}}{\tilde{U}_{q_k}}{U_{q_j}} + \tri{U_{q_j - 1}}{U_{q_j}}{U_{q_j}} \\
&-
\tri{U_{q_j}}{U_{q_j}}{U_{\leq q_{j-1}}} \\
& = A+B+C.
\end{align*}
Using Bernstein's inequalities we estimate
\begin{align*}
|A| & \lesssim  \lq |U_{q_j}|_\infty \sum_{k \geq j+1} | \tilde{U}_{q_k} |_2^2 \lesssim \lq^{2\a}\l_{q_{j+1}}^{4\a-5} \leq 1, \\
|C| & \lesssim |U_{q_j}|_2^2 \sum_{k \leq j-1} \lqk | \tilde{U}_{q_k} |_\infty \lesssim \l_{q_{j-1}}^{2\a}\lq^{4\a-5} \leq 1,
\end{align*}
where in the latter inequality we used the fact $|U_{q_j}|_2 \sim \lq^{2\a-5/2}$. On the other hand, a straightforward computation shows that
\begin{equation}
B \sim \lq^{6\a-5},
\end{equation}
which is thus the dominant term of the three, and hence,  
$$
\tri{U}{U}{U_{q_j}} \sim \lq^{6\a - 5}.
$$
\begin{proposition}\label{p-nse}
Let $u\in C_w([0,T); L^2) \cap L^2([0,T); H^1)$ be a weak solution to the NSE with initial condition $u(0) = U$. Then there is $\delta=\delta(u) >0$ such that
\begin{equation}
\limsup_{t \ra 0+} \|u(t) - U\|_{B^{1 - 2\a}_{\infty,\infty}} \geq \delta.
\end{equation}
If, in addition, $u$ is a Leray-Hopf solution satisfying the energy inequality \eqref{enineq}, then $\delta$ can be chosen independent of $u$.
\end{proposition}
\begin{proof}
Let us test \eqref{nse} with $u_{q_j}$. Using \eqref{extended}, we find
$$
\p_t( \tilde{u}_{q_j} \cdot u_{q_j}) = - \nu |\n|^\a \tilde{u}_{q_j} \cdot |\n|^\a u_{q_j} + \tri{u}{u}{u_{q_j}},
$$
where as defined before, $\tilde{u}_{q_j} = u_{q_j-1}+u_{q_j}+u_{q_j+1}$.
Denoting $E(t) = \int_0^t | |\n|^\a u|_2^{2} ds$ we obtain
\begin{multline}\label{crucial}
|\tilde{u}_{q_j}(t)|_2^2 \geq |U_{q_j}|_2^2 - \nu E(t) + c_1 \lq^{6\a-5} t \\
- c_2\int_0^t \left|  \tri{u}{u}{u_{q_j}} - \tri{U}{U}{U_{q_j}}  \right|_1 ds,
\end{multline}
for some positive constants $c_1$ and $c_2$.
We now show that if the conclusion of the proposition fails then for some small $t>0$ the integral term the growth of the integral term above becomes less than  $c_1\lq^{6\a-5} t$ for large $j$.
This forces $|\tilde{u}_{q_j}(t)|_2^2  \gtrsim \lq^{6\a-5} t$ for all large $j$. Hence $u$ has infinite energy, which is a contradiction.

So suppose that for every $\delta >0$ there exists $t_0 = t_0(\delta) >0$ such that $\|u(t)- U\|_{B^{1- 2 \a}_{\infty,\infty}} < \d$ for all $0<t\leq t_0$.
Denoting $w = u-U$ we write
\begin{multline*}
\tri{u}{u}{u_{q_j}} - \tri{U}{U}{U_{q_j}} = \tri{w}{U}{U_{q_j}} + \tri{u}{w}{U_{q_j}} \\ + \tri{u}{u}{w_{q_j}} = A + B+ C.
\end{multline*}
We will now decompose each triplet into three terms according to the type of interaction (c.f. Bony \cite{bony}) and estimate each of them separately.
\begin{multline*}
A = \sum_{\substack{p',p'' \geq q_j \\ |p' - p''| \leq 2}} \tri{w_{p'}}{U_{p''}}{U_{q_j}} + \tri{w_{\leq q_j}}{\tilde{U}_{q_{j}}}{ U_{q_j}} \\
+ \tri{\tilde{w}_{q_j}}{U_{\leq q_j}}{U_{q_j}} - \text{repeated} = A_1 + A_2 + A_3.
\end{multline*}
Let us fix $r\in(1,3/(4\alpha-2))$ and use Lemma~\ref{ubesov} along with H\"{o}lder and Bernstein's inequalities to estimate $A_1$:
\[
\begin{split}
| A_1 | \leq |\n U_{q_j}|_{r'} \sum |w_{p'}|_\infty |U_{p''}|_r \lesssim \lq^{2\a-3+\frac{3}{r}} \sum |w_{p'}|_\infty \l_{p''}^{2\a-1 - \frac{3}{r}} \\
\lesssim  \d \lq^{2\a-3+\frac{3}{r}} \leq \d \lq^{6\a - 5}.
\end{split}
\]
Intergrating by parts we obtain $A_2 = \tri{U_{q_j}}{\tilde{U}_{q_{j}}}{w_{\leq q_j} }$. Thus,  using the same tools,
\[
|A_2| \leq |\tilde{U}_{q_j}|_2^2 |\n w_{\leq q_j}|_\infty \lesssim
\lq^{4\a-5} \sum_{p\leq q_j} \l_p |w_p|_\infty < \d \lq^{6\a-5}.
\]
And finally,
\[
|A_3|  \leq \lq |U_{\leq q_j}|_2 |U_{q_j}|_{2} |\tilde{w}_{q_j}|_\infty \lesssim  \lq^{4\a-4}|\tilde{w}_{q_j}|_\infty < \d \lq^{6\a-5}.
\]
We have shown the following estimate:
\begin{equation}\label{A}
|A| \lesssim \d \lq^{6\a-5}.
\end{equation}
As to $B$ we decompose analogously,
\begin{multline*}
B = \sum_{\substack{p',p'' \geq q_j \\ |p' - p''| \leq 2}} \tri{u_{p'}}{w_{p''}}{U_{q_j}} + \tri{u_{\leq q_j}}{\tilde{w}_{q_{j}}}{ U_{q_j}} \\
+ \tri{\tilde{u}_{q_j}}{w_{\leq q_j}}{U_{q_j}} - \text{repeated} = B_1+B_2+B_3.
\end{multline*}
The term $B$ is the least problematic. Here we do not even have to use
the smallness of $w$ and can just roughly estimate it in terms of the enstrophy
$||\n|^\a u|_2^2$. We have
\[
\begin{split}
|B_1| & \lesssim \sum_{\substack{p',p'' \geq q_j \\ |p' - p''| \leq 2}} | \tri{u_{p'}}{u_{p''}}{U_{q_j}} | +   \sum_{\substack{p',p'' \geq q_j \\ |p' - p''| \leq 2}}  | \tri{u_{p'}}{U_{p''}}{U_{q_j}} | \\
&\leq \lq^{2\a } | u_{ \geq q_j } |_2^2 + \lq^{2\a } |u_{\geq q_j}|_2 |U_{\geq q_j} |_2
\\&\leq  | |\n|^\a u_{ \geq q_j } |_2^2 + \lq^{3\a - 5/2 } | |\n|^\a u_{\geq q_j}|_2 \\
&\leq  | |\n|^\a u_{ \geq q_j } |_2^2 + \lq^{6\a - 5 - 1/2 } | |\n|^\a u_{\geq q_j}|_2.
\end{split}
\]
Again, using Lemma~\ref{ubesov}, Bernstein and H\"{o}lder inequalities we obtain
\begin{align*}
|B_2| & =  \left| \tri{U_{q_j}}{\tilde{w}_{q_{j}}}{u_{\leq q_j}} \right| \leq  |U_{q_j}|_{2} |\tilde{w}_{q_{j}}|_\infty |\n u_{\leq q_j}|_2 \\
    & \leq \lq^{2\a-5/2} |\tilde{w}_{q_{j}}|_\infty ||\n|^{\a} u|_2 \leq \lq^{4\a-7/2} ||\n|^{\a} u|_2\leq  \lq^{6\a-5-1/2} ||\n|^{\a} u|_2.\\
|B_3| & \leq |\tilde{u}_{q_j}|_2 |w_{\leq q_j}|_\infty |\n U_{q_j}|_2 \lesssim \lq^{2\a-3/2} |\tilde{u}_{q_j}|_2  \sum_{p \leq q_j}  |w_p|_\infty \\
&\lesssim  \lq^{3\a-5/2} ||\n|^{\a} u|_2 \leq \lq^{6\a-5 - 1/2} ||\n|^{\a} u|_2 .
\end{align*}
We thus obtain
\begin{equation}\label{B}
|B| \lesssim | |\n|^\a u_{ \geq q_j } |_2^2 + \lq^{6\a - 5 - 1/2 } | |\n|^\a u|_2.
\end{equation}
Continuing in a similar fashion we write
\begin{multline*}
C = \sum_{\substack{p',p'' \geq q_j \\ |p' - p''| \leq 2}} \tri{u_{p'}}{u_{p''}}{w_{q_j}} + \tri{u_{\leq q_j}}{\tilde{u}_{q_{j}}}{w_{q_j}} \\
+ \tri{\tilde{u}_{q_j}}{u_{\leq q_j}}{w_{q_j}} - \text{repeated} = C_1+C_2+C_3.
\end{multline*}
We have
\[
|C_1|  \leq |\n w_{q_j}|_\infty | u_{\geq q_j} |_2^2 \lesssim \d | |\n|^\a u|_2^2.
\]
In $C_2$ we move the derivative onto $u_{\leq q_j}$ and estimate as usual,
\[
|C_2|  \leq | \n u|_{2} |\tilde{u}_{q_{j}}|_2 |w_{q_j}|_\infty \lesssim  | |\n|^\a u |_{2} |\tilde{u}_{q_{j}}|_2  \lq^{2\a - 1} \leq   | |\n|^\a u |^2_{2} \lq^{6\a - 5 - 1/2}.
\]
Using a uniform bound on the energy we have for $C_3$,
\[
|C_3| \lesssim \lq |w_{q_j}|_\infty |\tilde{u}_{q_j}|_2 \leq  \d \lq^{\a} | |\n|^\a \tilde{u}_{q_j}|_2 \leq \d \lq^{6\a-5} | |\n|^\a \tilde{u}_{q_j}|_2.
\]
Thus,
\begin{equation}\label{C}
|C| \lesssim \d | |\n|^\a u|_2^2 +  | |\n|^\a u |^2_{2} \lq^{6\a - 5 - 1/2} + \d \lq^{6\a-5} | |\n|^\a \tilde{u}_{q_j}|_2.
\end{equation}
Now combining estimates \eqref{A}, \eqref{B}, \eqref{C} along with the boundedness of $E(t_0)$ we obtain
\begin{equation} \label{nonlinearterm}
\begin{split}
\int_0^{t_0} \left| A+B+C \right| ds & \lesssim \d \lq^{6\a-5} t_0 +  \int_0^{t_0} | |\n|^\a u_{ \geq q_j } |_2^2\, ds + E(t_0)^{1/2} t_0^{1/2} \lq^{6\a - 5 - 1/2 } \\
& + \d E(t_0) +  \d \lq^{6\a-5} \int_0^{t_0} | |\n|^\a \tilde{u}_{q_j}|_2\, ds.
\end{split}
\end{equation}
And for large $j$, and fixed $t_0$, this gives
\[
\int_0^{t_0} \left| A+B+C \right| ds \lesssim \d \lq^{6\a-5} t_0 + \frac{\nu}{2} E(t_0).
\]
Pugging this back into \eqref{crucial} gives the estimate
$$
|\tilde{u}_{q_j}(t_0)|_2^2 \gtrsim \lq^{6\a - 5},
$$
for all $j >j_0$, which shows that $u(t_0)$ has infinite energy, a contradiction.

The last statement of the proposition follows from the fact that we have the bounds on $|u(t)|_2 \leq |U|_2$ and $E(t_0) \leq (2\nu)^{-1} |U|_2^2$ which remove dependence of the constants on $u$.
\end{proof}

%\bibliographystyle{plain}
%\bibliography{ill-fractionalNSE}
%\end{document}

\end{document}